\documentclass[10pt,12pt]{amsart}
\usepackage{amsmath}
\usepackage{amsfonts}
\usepackage{amssymb}

\newtheorem{theorem}{Theorem}

\newtheorem{corollary}[theorem]{Corollary}

\newtheorem{definition}[theorem]{Definition}

\input{tcilatex}

\begin{document}
\title[Continuous dependence on parameters]{Continuous dependence on
parameters for second order discrete BVP's}
\author{ Marek Galewski, Szymon G\l \c ab}
\maketitle

\begin{abstract}
Using min-max inequality we investigate the existence of solutions and thier
dependence on parameters for some second order discrete boundary value
problem. The approach is based on variational methods and solutions are
obtained as saddle points to the Euler action functional.
\end{abstract}

\section{Introduction}

Boundary value problems governed by discrete equations have received some
attention lately by both variational and topological approach. The
variational techniques applied for discrete problems include, among others,
the mountain pass methodology, the linking theorem, the Morse theory, the
three critical point, compare with \cite{agrawal}, \cite{caiYu}, \cite%
{AppMathLett}, \cite{TianZeng}, \cite{zhangcheng}, \cite{nonzero}. Moreover,
the fixed point approach is in fact much more prolific in the case of
discrete problem and covers the techniques already applied for continuous
problems, see for example \cite{FPAgrawal}, \cite{FPYangPing}, with both
list of references far from being exhaustive. \bigskip

While in the literature mainly the problem of the existence of solutions and
their multiplicity is considered, we are going to go a bit further and
investigate also the dependence on a functional parameter $u$ for the
following discrete boundary value problem which is a saddle -point type
system. Let $D>0$ be fixed. The problem which we consider reads 
\begin{equation}
\left\{ 
\begin{array}{ll}
\Delta ^{2}x(k-1)=F_{x}(k,x(k),y(k),u(k)),\bigskip  &  \\ 
\Delta ^{2}y(k-1)=-F_{y}(k,x(k),y(k),u(k)),\bigskip  &  \\ 
x(0)=x(T+1)=y(0)=y(T+1)=0, & 
\end{array}%
\right.   \label{zad}
\end{equation}%
where $F:[1,T]\times \mathbb{R}\times \mathbb{R}\times \left[ -D,D\right]
\rightarrow \mathbb{R}$ is a continuous function differentiable with respect
to the second and the third variable, 
\begin{equation*}
u\in L_{D}=\{u\in C([1,T],\mathbb{R}):||u||_{C}\leq D\},
\end{equation*}%
where $||u||_{C}$ denotes the classical maximum norm $||u||_{C}=\max_{k\in
\lbrack 1,T]}|u(k)|$ and $[a,b]$ for $a<b$, $a,b\in 
%TCIMACRO{\U{2124} }%
%BeginExpansion
\mathbb{Z}
%EndExpansion
$ denotes a discrete interval $\{a,a+1,...,b\}$. By a solution to (\ref{zad}%
) we mean a function $x:[0,T+1]\rightarrow \mathbb{R}$\ which satisfies the
given equation and the associated boundary conditions.\bigskip 

Such type of a difference equation as (\ref{zad}) may arise from evaluating
the Dirichlet boundary value problem%
\begin{equation*}
\begin{array}{l}
\frac{d^{2}}{dt^{2}}x=G_{x}\left( t,x,y,u\right) \text{,}\bigskip \text{ }%
\frac{d^{2}}{dt^{2}}y=-G_{y}\left( t,x,y,u\right) \text{,} \\ 
0<t<1\text{, }x\left( 0\right) =x(1)=0,y\left( 0\right) =y(1)=0%
\end{array}%
\end{equation*}%
where $G:\left[ 0,1\right] \times \mathbb{R}\times \mathbb{R}\times \mathbb{R%
}\rightarrow \mathbb{R}$ is continuous and subject to some growth
conditions. Such a continuous problem subject to a functional parameter has
been considered in \cite{JS}. \ \bigskip 

The question whether the system depends continuously on a parameter is vital
in context of the applications, where the measurements are known with some
accuracy. This question is even more important when \ the solution to the
problem under consideration is not unique as is the case of the present
note. In the boundary value problems for differential equations there are
some results towards the dependence of a solution on a functional parameter,
see \cite{LedzewiczWalczak}, \cite{JS} with references therein. This is not
the case with discrete equations where we have only some results which use
the critical point theory, see \cite{w}. The approach of this note is
different from this of \cite{w} since it does not relay on coercivity
arguments but on a min-max inequality due to Ky Fan, see \cite{nirenberg}.
In our approach we use some ideas developed in \cite{JS} suitable modified
due to the finite dimensionality of the space under consideration.

The following results will be used in the sequel, see \cite{nirenberg}.

\begin{theorem}[Fan's Min--Max Theorem]
\label{kYfanTh}Let $X$ and $Y$ be Hausdorff topological vector spaces, $%
A\subset X$ and $B\subset Y$ be convex sets, and $J:A\times B\rightarrow 
\mathbb{R}$ be a function which satisfies the following conditions:

\begin{itemize}
\item[(i)] for each $y\in B$, the functional $x\rightarrow J(x,y)\rightarrow 
\mathbb{R}$ convex and lower semi-continuous on $A$;

\item[(ii)] for each $x\in A$, the functional $y\rightarrow
J(x,y)\rightarrow \mathbb{R}$ is concave and upper semi-continuous on $B$;

\item[(iii)] for some $x_0\in A$ and some $\delta_0<\inf_{x\in A}\sup_{y\in
B}J(x,y)$, the set $\{y\in B:J(x_0,y)\}$ is compact.
\end{itemize}

Then 
\begin{equation*}
\sup_{y}\inf_{x}J(x,y)=\inf_{x}\sup_{y}J(x,y).
\end{equation*}
%\bigskip\ of all saddle points of $J_{u}$ is contained in $B_{1}\times B_{2}$%.
\end{theorem}

\begin{definition}
Let $(X,\tau )$ be a Hausdorff topological space and let $%
(A_{n})_{n=1}^{\infty }$ be a sequence of nonempty subsets of $X$. The set
of accumulation points of sequences $(a_{n})_{n=1}^{\infty }$ with $a_{n}\in
A_{n}$ for $n=1,2,3,...$ is called the upper limit of $(A_{n})_{n=1}^{\infty
}$ and denoted by $\limsup A_{n}$.
\end{definition}

\section{Variational framework for problem (\protect\ref{zad})}

Solutions to (\ref{zad}) will be investigated in the space 
\begin{equation*}
H=\{x:[0,T+1]\rightarrow \mathbb{R}:x(0)=x(T+1)=0\}
\end{equation*}
considered with the norm 
\begin{equation*}
||x||=\left( \sum_{k=1}^{T+1}|\Delta x(k-1)|^{2}\right) ^{1/2}.
\end{equation*}%
Then $(H,||\cdot ||)$ becomes a Hilbert space. For any $m\geq 2$ let $c_{m}$
be the smallest positive constant such that 
\begin{equation*}
\sum_{k=1}^{T}|x(k)|^{m}\leq c_{m}\cdot \sum_{k=1}^{T+1}|\Delta x(k-1)|^{m}
\end{equation*}%
for any $x\in H$; see \cite[Lemma 1]{MRT}. \bigskip

Since the approach of present note is a variational one we investigate the
action functional $J_{u}:H\times H\rightarrow R$, corresponding to problem (%
\ref{zad}). For a fixed parameter $u\in L_{D}$, $J_{u}$ is of the form 
\begin{equation*}
J_{u}(x,y)=\sum_{k=1}^{T+1}\frac{|\Delta x(k-1)|^{2}}{2}-\frac{|\Delta
y(k-1)|^{2}}{2}+\sum_{k=1}^{T}F(k,x(k),y(k),u(k)).
\end{equation*}

We assume that $F$ has the following properties:\bigskip

\begin{itemize}
\item[H1] $F:[1,T]\times \mathbb{R}\times \mathbb{R}\times \mathbb{R}%
\rightarrow \mathbb{R}$ is a continuous function which is differentiable
with respect to the second and the third variable; $F_{x},F_{y}:[1,T]\times 
\mathbb{R}\times \mathbb{R}\times \mathbb{R}\rightarrow \mathbb{R}$ are
continuous functions. \bigskip 

\item[H2] For any fixed $y\in H$ there are a constant $\beta _{1},$ a
function $\gamma _{1}:\left[ 1,T\right] \rightarrow \mathbb{R}$ and a
constant $\alpha _{1}<1/(2c_{2})$ such that 
\begin{equation*}
F(k,x,y(k),u)\geq -\alpha _{1}|x|^{2}+\beta _{1}x+\gamma _{1}(k)
\end{equation*}%
for all $x\in \mathbb{R}$, all $u\in \mathbb{R}$, $\left\vert u\right\vert
\leq D$ and all $k\in \left[ 1,T\right] .$\bigskip 

\item[H3] For any fixed $x\in H$ there are a constant $\beta _{2},$ a
function $\gamma _{2}:\left[ 1,T\right] \rightarrow \mathbb{R}$ and a
constant $\alpha _{2}<1/(2c_{2})$ such that 
\begin{equation*}
F(k,x(k),y,u)\leq \alpha _{2}|y|^{2}+\beta _{2}y+\gamma _{2}(k)
\end{equation*}%
for all $y\in \mathbb{R}$, all $u\in \mathbb{R}$, $\left\vert u\right\vert
\leq D$ and all $k\in \left[ 1,T\right] .$\bigskip 

\item[H4] Functional $x\rightarrow J_{u}(x,y)$ is convex for all $y\in H$, $%
u\in L_{D}$.\bigskip 

\item[H5] Functional $y\rightarrow J_{u}(x,y)$ is concave for all $x\in H$, $%
u\in L_{D}$.\bigskip 
\end{itemize}

We observe that with any fixed $u\in L_{D}$ functional $J_{u}$ is
continuous. With the aid of Theorem \ref{kYfanTh} we are able to find saddle
points for functional $J_{u}$. Since $J_{u}$ is differentiable in the sense
of G\^{a}teaux, it is apparent that such points are the critical points to $%
J_{u}$. Since in turn critical points to $J_{u}$ constitute solutions to (%
\ref{zad}), we arrive at existence result once we get the existence of
saddle points. Moreover, since the spaces in which we work are finite
dimensional one, there is no need to distinguish between the weak and the
strong solutions.

\section{Existence of saddle point solutions}

\begin{theorem}[Existence of saddle points]
\label{ExistSaddlePoints}Assume that conditions H1-H2 hold. Let $u\in L_{D}$
be fixed. Then it follows that\newline
(A) There is a saddle point $(x_{u},y_{u})$ for the functional $J$;\newline
(B) There are balls $B_{1}=\{x:||x||\leq r_{1}\}$ and $B_{2}=\{y:||y||\leq
r_{2}\}$ such that $(x_{u},y_{u})\in B_{1}\times B_{2}$;\newline
(C) The set of all saddle points of $J_{u}$ is compact.
\end{theorem}

\begin{proof}
For fixed $y\in H$ using H2 we obtain%
\begin{equation*}
\begin{array}{l}
J_{u}(x,y)\geq \sum_{k=1}^{T+1}\left( \frac{|\Delta x(k-1)|^{2}}{2}-\frac{%
|\Delta y(k-1)|^{2}}{2}-\alpha _{1}|x(k)|^{2}+\beta _{1}x(k)+\gamma
_{1}\right) \geq \bigskip  \\ 
\left( \frac{1}{2}-c_{2}\alpha _{1}\right) ||x||^{2}+\tilde{\beta}_{1}||x||+%
\tilde{\gamma}_{1},%
\end{array}%
\end{equation*}
where $\tilde{\beta}_{1}>0$ depends only on $\beta _{1}$ (note that $||x||$
and $\sum_{k=1}^{T+1}|x(k)|$ are equivalent norms, since $H$ is
finite-dimensional) and $\tilde{\gamma}_{1}>0$ depends only on $\gamma _{1}$
and $y$. Since $\frac{1}{2}-c_{2}\alpha _{1}>0$, the functional $J_{u}(x,y)$
is coercive on $H$. By H1 and H4 it is continuous and convex for each $u$.
Put 
\begin{equation*}
J_{u}^{-}(y)=\min_{x}J_{u}(x,y).
\end{equation*}%
By H5 the functional $J_{u}^{-}$ is concave. By H3 we obtain that 
\begin{equation}
\begin{array}{l}
J_{u}^{-}(y)\leq J_{u}(0,y)\leq \bigskip  \\ 
\sum_{k=1}^{T+1}\left( -\frac{|\Delta y(k-1)|^{2}}{2}+\alpha
_{2}|y(k)|^{2}+\beta _{2}y(k)+\gamma _{2}\right) \leq \bigskip  \\ 
\left( -\frac{1}{2}+c_{2}\alpha _{1}\right) ||y||^{2}+\tilde{\beta}_{2}||y||+%
\tilde{\gamma}_{2},%
\end{array}
\label{eq1}
\end{equation}%
where $\tilde{\beta}_{2}>0$ depends only on $\beta _{2}$ and $\tilde{\gamma}%
_{2}>0$ depends only on $\gamma _{2}$. Since the constant $-\frac{1}{2}%
+c_{2}\alpha _{1}$ is negative, then $J_{u}^{-}$ is anti-coercive. Hence it
attains its supremum at some point $y_{u}$. By H2 we have%
\begin{equation*}
\begin{array}{l}
J_{u}^{-}(y_{u})\geq J_{u}^{-}(0)=\min_{x}J_{u}(x,0)\geq \bigskip  \\ 
\min_{x}\left( \left( \frac{1}{2}-c_{2}\alpha _{1}\right) ||x||^{2}+%
\overline{\beta }_{1}||x||+\gamma _{1}\right) =\gamma _{1}.%
\end{array}%
\end{equation*}
Since $J_{u}^{-}$ is anti-coercive, there is $r_{2}>0$ such that $%
J_{u}^{-}(y)<\gamma _{1}$ for every $||y||>r_{2}$. Since $J_{u}^{-}$ is
continuous the set $\{y:J_{u}^{-}(y)\geq \gamma _{1}\}$ is compact and is
contained in $B_{2}$. Hence each $y_{u}$ is in $B_{2}$.\bigskip 

Analogously one can show that there is $x_{u}$ with 
\begin{equation*}
J_{u}^{+}(x_{u})=\min_{x}J_{u}^{+}=\min_{x}\max_{y}J_{u}(x,y).
\end{equation*}%
Furthermore, there is a ball $B_{1}$ with $x_{u}\in B_{1}$ for each such $%
x_{u}$.

We have already showed that for each $x$ there exists $\max_{y}J_{u}(x,y)$.
Hence for some $\delta _{0}$ we have 
\begin{equation*}
\delta _{0}<\min_{x}J_{u}(x,0)\leq \min_{x}\max_{y}J_{u}(x,y).
\end{equation*}%
By (\ref{eq1}) we obtain 
\begin{equation*}
\{y:J_{u}(0,y)\geq \delta _{0}\}\subset \{y:(-1/2+c_{2}\alpha _{1})||y||^{2}+%
\tilde{\beta}_{2}||y||+\tilde{\gamma}_{2}\geq \delta _{0}\}.
\end{equation*}%
Since the set of right hand of inclusion is compact, so is the set $%
\{y:J_{u}(0,y)\geq \delta _{0}\}$. Thus, the assumptions H4 and H5 and Fan's
minimax Theorem \ref{kYfanTh}, give the existence of a saddle point of $%
J_{u} $. Moreover the set of all saddle points of $J_{u}$ is compact.
\end{proof}

\begin{theorem}[Existence of saddle point solutions]
\label{ExistSaddlePointSol}Assume that conditions H1-H5 hold. Let $u\in
L_{D} $ be fixed. Then it follows that there exists is at least one saddle
point $(x_{u},y_{u})\in H\times H$ for the functional $J_{u}$ which solves (%
\ref{zad}).
\end{theorem}

\begin{proof}
By Theorem \ref{ExistSaddlePoints} there is at least one saddle point $%
(x_{u},y_{u})$ for the functional $J_{u}$. Since $J_{u}$ is a G\^{a}teaux
differentiable functional we see that $J_{u}^{^{\prime }}(x_{u},y_{u})=0$
and therefore $(x_{u},y_{u})$ solves (\ref{zad}).
\end{proof}

In order to obtain existence results we do not need to impose conditions
H2-H5 uniformly in $u$. This is not the case when one is interested in the
dependence on parameters, when assumptions must be placed uniformly with
respect to $u$. Indeed, let us consider a following problem 
\begin{equation}
\left\{ 
\begin{array}{l}
\Delta ^{2}x(k-1)=F_{x}(k,x(k),y(k)),\bigskip  \\ 
\Delta ^{2}y(k-1)=-F_{y}(k,x(k),y(k)),\bigskip  \\ 
x(0)=x(T+1)=y(0)=y(T+1)=0,%
\end{array}%
\right.   \label{zad2}
\end{equation}%
where $F:[1,T]\times \mathbb{R}\times \mathbb{R}\rightarrow \mathbb{R}$ is a
continuous function which is differentiable with respect to the second and
the third variable. The action functional $J:H\times H\rightarrow R$,
corresponding to problem (\ref{zad2}) is 
\begin{equation*}
J(x,y)=\sum_{k=1}^{T+1}\frac{|\Delta x(k-1)|^{2}}{2}-\frac{|\Delta
y(k-1)|^{2}}{2}+\sum_{k=1}^{T}F(k,x(k),y(k)).
\end{equation*}%
We assume that\bigskip 

\begin{itemize}
\item[H6] $F:[1,T]\times \mathbb{R}\times \mathbb{R}\rightarrow \mathbb{R}$
is a continuous function which is differentiable with respect to the second
and the third variable; $F_{x},F_{y}:[1,T]\times \mathbb{R}\times \mathbb{R}%
\rightarrow \mathbb{R}$ are continuous functions.\bigskip 

\item[H7] For any fixed $y\in H$ there are a constant $\beta _{1},$ a
function $\gamma _{1}:\left[ 1,T\right] \rightarrow \mathbb{R}$ and a
constant $\alpha _{1}<1/(2c_{2})$ such that 
\begin{equation*}
F(k,x,y(k),u)\geq -\alpha _{1}|x|^{2}+\beta _{1}x+\gamma _{1}(k)
\end{equation*}%
for all $x\in \mathbb{R}$ and all $k\in \left[ 1,T\right] .$\bigskip 

\item[H8] For any fixed $x\in H$ there are a constant $\beta _{2},$ a
function $\gamma _{2}:\left[ 1,T\right] \rightarrow \mathbb{R}$ and a
constant $\alpha _{2}<1/(2c_{2})$ such that 
\begin{equation*}
F(k,x(k),y,u)\leq \alpha _{2}|y|^{2}+\beta _{2}y+\gamma _{2}(k)
\end{equation*}%
for all $y\in \mathbb{R}$ and all $k\in \left[ 1,T\right] .$\bigskip 

\item[H9] Functional $x\rightarrow J_{u}(x,y)$ is convex for any $y\in H$%
.\bigskip 

\item[H10] Functional $y\rightarrow J_{u}(x,y)$ is concave for any $x\in H$%
.\bigskip 
\end{itemize}

Then we have

\begin{corollary}
Assume that conditions H6-H10 hold. Then it follows that there exists is at
least one saddle point $(x,y)\in H\times H$ for the functional $J$ which
solves (\ref{zad}).
\end{corollary}

\section{Continuous dependence on parameters}

Now we are interested of the behavior of the sequence of saddle points which
correspond to a sequence of parameters. Dependence on parameters in
investigated through the convergence of the sequence of action functionals
corresponding the sequence of parameters - this approach has already been
applied with some success for the continuous and also the discrete problems,
see \cite{w}, \cite{LedzewiczWalczak}. Let $(u_{n})_{n=1}^{\infty }\subset
L_{D}$ be a sequence of parameters. We put $J_{n}=J_{u_{n}}$ and let 
\begin{equation*}
V_{n}=\{(\overline{x},\overline{y}):J_{n}(\overline{x},\overline{y}%
)=\max_{y}\min_{x}J_{n}(x,y)\}\subset B_{1}\times B_{2}
\end{equation*}%
be the set of all saddle points of $J_{n}$. Due to Theorem \ref%
{ExistSaddlePoints} $V_{n}\neq \emptyset $ for all $n=1,2,...$ .

\begin{theorem}
\label{condepd}Assume that conditions H1-H5 hold. Let $(u_{n})_{n=1}^{\infty
}\subset L_{D}$ be a \ convergent sequence of parameters and $%
u_{n}\rightarrow u_{0}\in L_{D}$ as $n\rightarrow \infty $. Then $\emptyset
\neq \limsup_{n\rightarrow \infty }V_{n}\subset V_{0}$.
\end{theorem}

\begin{proof}
At first we observe by continuity of $F$ that $J_{n}$ tends to $J_{0}$
uniformly on $B_{1}\times B_{2}$, where $B_{1},$ $B_{2}$ are defined in
Theorem \ref{ExistSaddlePoints}. We will prove that $\emptyset \neq \limsup
V_{n}\subset V_{0}$. Let $a_{n}=\max_{y}\min_{x}J_{n}(x,y)$ and let $%
\varepsilon >0$. Since $J_{n}$ tends uniformly to $J_{0}$, then $%
J_{n}(x,y)\leq J_{0}(x,y)+\varepsilon $ for each $(x,y)\in B_{1}\times B_{2}$
and every $n\geq n_{0}$ for some $n_{0}$. Then 
\begin{equation*}
\min_{x}J_{n}(x,y)\leq \min_{x}J_{0}(x,y)+\varepsilon ,
\end{equation*}%
\begin{equation*}
\max_{y}\min_{x}J_{n}(x,y)\leq \max_{y}\min_{x}J_{0}(x,y)+\varepsilon .
\end{equation*}%
Hence $a_{k}-a_{0}\leq \varepsilon $. Similarly one can show that $%
a_{k}-a_{0}\geq -\varepsilon $. Therefore $a_{k}\rightarrow a_{0}$.\bigskip 

Let $(x_{n},y_{n})\in V_{n}$ for $n=1,2,...$. Since 
\begin{equation*}
\{(x_{n},y_{n})\}_{n=1}^{\infty }\subset B_{1}\times B_{2}
\end{equation*}%
we may assume that $(x_{n},y_{n})\rightarrow (x_{0},y_{0})$. In particular $%
\limsup V_{n}\neq \emptyset $. Suppose now that $(x_{0},y_{0})\notin V_{0}$.
Let $(\overline{x},\overline{y})\in V_{0}$. Then $J_{0}(\overline{x},%
\overline{y})\neq J_{0}(x_{0},y_{0})$. Consider the case 
\begin{equation*}
J_{0}(\overline{x},\overline{y})-J_{0}(x_{0},y_{0})=\eta <0.
\end{equation*}%
Then%
\begin{equation*}
\begin{array}{l}
a_{n}-a_{0}=J_{n}(x_{n},y_{n})-J_{0}(x_{0},y_{0})=\bigskip  \\ 
\min_{x}J_{n}(x,y_{n})-J_{0}(x_{0},y_{0})\leq \bigskip  \\ 
\leq J_{n}(\overline{x},y_{n})-J_{0}(x_{0},y_{0})=\bigskip  \\ 
J_{n}(\overline{x},y_{n})-J_{0}(\overline{x},y_{n})+J_{0}(\overline{x}%
,y_{n})-J_{0}(\overline{x},\overline{y})+J_{0}(\overline{x},\overline{y}%
)-J_{0}(x_{0},y_{0}).%
\end{array}%
\end{equation*}
Since%
\begin{equation*}
J_{0}(\overline{x},\overline{y})=\max_{y}J_{0}(\overline{x},y)\geq J_{0}(%
\overline{x},y_{n}),
\end{equation*}%
then 
\begin{equation*}
\limsup_{n}J_{0}(\overline{x},y_{n})-J_{0}(\overline{x},\overline{y})\leq 0.
\end{equation*}%
By the continuity of $F$ we obtain that $J_{n}(\overline{x}%
,y_{n})\rightarrow J_{0}(\overline{x},y_{n})$. Therefore 
\begin{equation*}
\limsup_{n\rightarrow \infty }(a_{n}-a_{0})<\eta .
\end{equation*}%
A contradiction. Similarly, a contradiction can be obtained when $\eta >0$.
\end{proof}

Theorem \ref{condepd} combined with Theorem \ref{ExistSaddlePointSol} yield
the following main result of our note

\begin{theorem}
Assume H1-H5. For any fixed $u\in L_{D}$ there exists at least one solution $%
y\in V_{u}$ to problem (\ref{zad}). Let $\{u_{n}\}\subset L_{D}$ be a
convergent sequence of parameters, where $\underset{n\rightarrow \infty }{%
\lim }u_{n}=u_{0}\in L_{D}$. For any sequence $\{\left( x_{n},y_{n}\right) \}
$ of solutions $\left( x_{n},y_{n}\right) \in V_{n}$ to the problem (\ref%
{zad}) corresponding to $u_{n}$, there exist a subsequence $\{\left(
x_{n_{i}},y_{n_{i}}\right) \}\subset H\times H$ and an element $\left(
x_{0},y_{0}\right) \subset H\times H$ such that $\underset{i\rightarrow
\infty }{\lim }x_{n_{i}}=x_{0}$, $\underset{i\rightarrow \infty }{\lim }%
y_{n_{i}}=y_{0}$ and $J_{0}(x_{0},y_{0})=\max_{y}\min_{x}J_{0}(x,y)$.
Moreover $x_{0},y_{0}\in V_{0}$, i.e. the pair $\left( x_{0},y_{0}\right) $
satisfies (\ref{zad}) with $u=u_{0}$, namely 
\begin{equation*}
\left\{ 
\begin{array}{l}
\Delta ^{2}x_{0}(k-1)=F_{x}(k,x_{0}(k),y_{0}(k),u_{0}(k)),\bigskip  \\ 
\Delta ^{2}y_{0}(k-1)=-F_{y}(k,x_{0}(k),y_{0}(k),u_{0}(k)),\bigskip  \\ 
x_{0}(0)=x_{0}(T+1)=y_{0}(0)=y_{0}(T+1)=0.%
\end{array}%
\right. 
\end{equation*}
\end{theorem}

\begin{tabular}{l}
Marek Galewski, Szymon G\l \c{a}b\smallskip  \\ 
Institute of Mathematics,\smallskip  \\ 
Technical University of Lodz,\smallskip  \\ 
Wolczanska 215, 90-924 Lodz, Poland,\smallskip  \\ 
marek.galewski@p.lodz.pl,\smallskip  \\ 
szymon.glab@p.lodz.pl\smallskip 
\end{tabular}

\end{document}